\documentclass[a4paper,12pt]{amsart}

\usepackage[utf8x]{inputenc}
\usepackage[english]{babel}

\usepackage{textcase}
\usepackage{hyperref}
\usepackage{setspace,amsmath,amssymb,amscd,amsthm,amsfonts}

\usepackage[left=2.5cm,right=2.5cm,top=2.5cm,bottom=2.5cm]{geometry}
\usepackage{xcolor}

\usepackage{import, xifthen, pdfpages, transparent}

\newcommand{\K}{\mathcal K}
\newcommand{\RR}{\mathbb R}
\newcommand{\CC}{\mathbb C}
\newcommand{\ZZ}{\mathbb Z}
\renewcommand{\le}{\leqslant}
\renewcommand{\ge}{\geqslant}

\newcommand{\R}{\mathcal R}
\newcommand{\kk}{\mathbf{k}}

\newtheorem{formula}{}[section]
\newtheorem{corollary}[formula]{Corollary}
\newtheorem{lemma}[formula]{Lemma}
\newtheorem{theorem}[formula]{Theorem}
\newtheorem{proposition}[formula]{Proposition}
\theoremstyle{definition}
\newtheorem{definition}[formula]{Definition}
\newtheorem{example}[formula]{Example}
\newtheorem{construction}[formula]{Construction}
\theoremstyle{remark}
\newtheorem*{remark}{Remark}

\numberwithin{equation}{section}

\DeclareMathOperator{\conv}{conv}
\DeclareMathOperator{\sk}{sk}

\DeclareMathOperator{\Perm}{Perm}
\DeclareMathOperator{\relint}{relint}
\DeclareMathOperator{\Tor}{Tor}
\DeclareMathOperator{\mdeg}{mdeg}

\begin{document}
\makeatletter
    \def\@@and{and}
\makeatother

\title[Permutohedral complex and complements of diagonal arrangements]{Permutohedral complex and complements of diagonal subspace arrangements}
\author{Taras Panov}
\address{Faculty of Mathematics and Mechanics, Moscow State University, Moscow, Russia;\newline
Institute for Information Transmission Problems (Kharkevich Institute), Russian Academy of Sciences, Moscow, Russia;\newline
Steklov Mathematical Institute of Russian Academy of Sciences, Moscow, Russia;\newline
National Research University Higher School of Economics, Moscow, Russia}
\email{\href{mailto:tpanov@mi-ras.ru}{tpanov@mi-ras.ru}}

\author{Vsevolod Tril}
\address{Faculty of Mathematics and Mechanics, Moscow State University, Moscow, Russia;\newline
National Research University Higher School of Economics, Moscow, Russia}
\email{\href{mailto:vsevolod.tril@math.msu.ru}{vsevolod.tril@math.msu.ru}}
\thanks{The work was supported by the Russian Science Foundation under grant no.~23-11-00143, https://rscf.ru/project/23-11-00143/
The second author is a recipient of a scholarship from the Theoretical Physics and Mathematics Advancement Foundation ``BASIS''
}
% 57S12      Toric topology
% 14N20  	Configurations and arrangements of linear subspaces
% 55N45  	Products and intersections in homology and cohomology
% 57T30  	Bar and cobar constructions

\begin{abstract}
The complement of an arrangement of diagonal subspaces $x_{i_1} = \cdots = x_{i_k}$ in the real space is defined by a simplicial complex~$\K$.
In this paper, we prove that the complement of a diagonal subspace arrangement is homotopy equivalent to a subcomplex $\Perm(\K)$ of faces of the permutohedron. The product in the cohomology ring of the complement of a diagonal arrangement is then described via Saneblidze and Umble's cellular approximation of the diagonal map in the permutohedron. We consider the projection from the permutohedron to the cube and prove that the Saneblidze--Umble diagonal maps to the diagonal constructed by Li Cai for describing the product in the cohomology of a real moment-angle complex.
\end{abstract}

\maketitle

\section{Introduction}
Subspace arrangements play an important role in many constructions of combinatorics, algebraic and symplectic geometry, as well as in the theory of mechanical systems. They first appeared in Arnold's paper~\cite{Arn69}, where he introduced the arrangement of complex hyperplanes $\{z_i = z_j \}$ and shown that it is the Eilenberg--Mac Lane space of the pure braid group. 
The cohomology ring of this space was also described therein.

Two classes of arrangements are of particular interest.
The first class is the coordinate subspace arrangements, which have been extensively studied in toric topology. The complements of coordinate subspace arrangements in $m$-dimensional space are encoded by simplicial complexes $\K$ on the set of $m$ elements. In~\cite{BP00} it was proved that the complement of a complex (respectively, real) coordinate subspace arrangement is homotopy equivalent to the moment-angle complex $\mathcal Z_{\mathcal K}$ (respectively, real moment-angle complex $\mathcal R_{\mathcal K}$). The same paper presented a dga model for the cohomology ring of the complement of a complex coordinate arrangement, constructed using a cellular decomposition of the moment-angle complex. For some classes of simplicial complexes, the homotopy type of these spaces was described explicitly.

The multiplicative structure in the cohomology of real moment-angle complexes $\mathcal R_{\mathcal K}$ was studied by Li Cai~\cite{LiCai}. Using the standard cellular decomposition of the cube, he constructed a dga-model for the cohomology ring of a real moment-angle complex.

Another important class is the diagonal subspace arrangements, whose complements are also encoded by simplicial complexes~$\mathcal K$.
It was introduced in~\cite{PRW}, where the cohomology groups of real diagonal arrangements were calculated via the bar construction of the Stanley--Reisner ring.
These results were further developed in the works~\cite{BP00} and~\cite{Dobr09}, where a connection was established between the complements of diagonal arrangements and loop spaces of polyhedral products (in particular, of moment-angle complexes). Moreover, it was proved in~\cite{T} that for a certain family of simplicial complexes, the corresponding moment-angle complex is homotopy equivalent to the double suspension of the complement of the diagonal subspace arrangement.

In this paper we prove that the complement of the diagonal arrangement corresponding to a simplicial complex~$\K$ is homotopy equivalent to a subcomplex $\Perm(\K)$ of faces in the permutohedron, and describe the combinatorics of this complex.
Analyzing the structure of the cellular chain complex of $\Perm(\K)$, we prove the following theorem.

\begin{theorem}\label{theorem-intro1}
Let $\K$ be a simplicial complex on the vertex set $[m]$, and let $D_\RR(\K)$ be the corresponding complement of a diagonal arrangement.
Then for any commutative ring $\kk$ with unit, there is an isomorphism of $\kk$-modules
\[
  H^q(D_\RR(\K); \kk) \cong \Tor_{\Lambda[\K]}^{q-m}(\kk, \kk)_{(1, \ldots, 1 )},
\]
where $\Lambda[\K]$ is the exterior Stanley--Reisner algebra of~$\K$.
\end{theorem}

To describe the product in the cohomology ring $H^*(D_\RR(\K))$,
we use the cellular diagonal approximation of the permutohedron,  $\Delta_{SU}$, constructed by Saneblidze and Umble~\cite{SU}. In~\cite{SU}, a projection $\rho \colon \Perm^{m-1} \to I^{m-1}$ from the permutohedron to the cube was also constructed. We prove that this projection maps the Saneblidze--Umble diagonal to the Li Cai diagonal~\cite{LiCai}.

\begin{theorem}\label{theorem-intro2}
Let $\rho_* \colon C_*(\Perm^{m-1}) \to C_*(I^{m-1})$ be the homomorphism of cellular chain complexes induced by the projection from the permutohedron to the cube.
Then for each chain $F \in C_*(\Perm^{m-1})$ corresponding to a face of the permutohedron, we have
\[
  (\rho_* \otimes \rho_*)\Delta_{SU}(F) = \Delta_{C}(\rho_* F),
\]  
where $\Delta_{C}$ is the diagonal on $C_*(I^{m-1})$ dual to the product described in~\cite{LiCai}.
\end{theorem}

The work is organized as follows. In Section 2 we review the main combinatorial objects related to diagonal arrangements. Section~3 provides the definition of a real moment-angle complex and the dga-model of its cohomology ring constructed in~\cite{LiCai}.
In Section~4 we introduce the complex $\Perm(\K)$ and prove that it is homotopy equivalent to the complement of a diagonal subspace arrangement.
In Section~5 we prove Theorem~\ref{theorem-intro1} by constructing an isomorphism of the cellular cochain complex $C^*(\Perm(\K); \kk)$ and the graded component of the bar construction of the Stanley--Reisner ring $\Lambda[\K]$.
Then in Section~6 we describe the diagonal constructed in~\cite{SU}.
In Section~7 we prove Theorem~\ref{theorem-intro2} and show that the image of the complex $\Perm(\K)$ under the projection $\rho$ is a real moment-angle complex.

The authors are grateful to the anonymous referee for useful comments and suggestions.

\section{Subspace arrangements}
{\it An arrangement} is a finite set $\mathcal{A} = \{L_1, \ldots, L_r \}$ of affine subspaces in some affine space (either real or complex).

An arrangement $\mathcal{A} = \{L_1, \ldots, L_r \}$ is called \textit{coordinate}
if every subspace $L_i$, $i = 1, \ldots, r$, is a coordinate subspace.
Every coordinate subspace in $\RR^m$ can be written as
\[
  C_I = \{(x_1, \ldots, x_m) \in \RR^m \colon x_{i_1} = \cdots = x_{i_k} = 0 \}, 
\]  
where $I = \{i_1, \ldots, i_k \}$ is a subset of $[m] = \{1, 2, \ldots, m \}$.

A \textit{simplicial complex} on a set $V$ is a collection $\K$ of subsets $\sigma \subset V$, called \emph{simplices}, 
that satisfies the condition: if $\sigma \in \K$ and $\tau \subset \sigma$ then $\tau \in \K$.

One-element simplices $\{v\} \in \K$ are called \textit{vertices} of $\K$. If all one-element subsets $\{v\} \subset V$ are vertices of $\K$,
we say that $\K$ is a simplicial complex on the \emph{vertex set} $V$. Usually, $V$ is the set $[m] = \{1, 2, \ldots, m \}$.

For a simplicial complex $\K$ on the set $[m]$, we define the corresponding \textit{real coordinate arrangement} $\mathcal{CA}(\K)$ and its \emph{complement}
$U_{\RR}(\K)$:
\[
  \mathcal{CA}(\K) = \{C_I \colon I \notin \K \},\qquad
  U_{\RR}(\K) = \RR^m \backslash \bigcup_{I \notin \K} C_I.
\]

\begin{proposition}[\cite{BP00}, Proposition~5.2.2]
The assignment $\K \mapsto U_\RR(\K)$ defines a one-to-one order-preserving correspondence between the set
of simplicial complexes on $[m]$ and the set of complements of coordinate arrangement in~$\RR^m$.
\end{proposition}

For each subset $I = \{ i_1, \ldots, i_k\} \subset [m]$, we define the \textit{diagonal subspace} $D_I$ in $\RR^m$ by
\[
 D_I = \{(x_1, \ldots, x_m) \in \RR^m \colon x_{i_1} = \cdots = x_{i_k} \}. 
\] 
An arrangement $\mathcal{A} = \{L_1, \ldots, L_r \}$ is called \emph{diagonal} if every subspace $L_i$, $i = 1, \ldots, r$,
is a diagonal subspace.

Given a simplicial complex $\K$ on the vertex set $[m]$, define the \textit{real diagonal arrangement} $\mathcal{DA}(\K)$
as the collection of subspaces $D_I$ such that $I$ is not a simplex in $\K$, and consider its \emph{complement} $D_\RR(\K)$:
\[
  \mathcal{DA}(\K) = \{D_I \colon I \notin \K \},\qquad
  D_\RR(\K)=\RR^m \backslash \bigcup_{I \notin \K} D_I.
\]

Complex diagonal subspaces $D_I^\CC \subset \CC^m$, diagonal arrangements $\mathcal{DA}_\CC(\K)$, and their complements $D_\CC(\K)$
are defined similarly.

\begin{proposition}[\cite{BP00}, Proposition~5.3.2]
The assignment $\K \mapsto D_\RR(\K)$ defines a one-to-one order-preserving correspondence between the set 
of simplicial complexes on the vertex set $[m]$ and the set of diagonal arrangement complements in $\RR^m$.
\end{proposition}

\section{Real moment-angle complexes and the Cai diagonal}
The following class of topological spaces, called moment-angle complexes, is useful for studying the complements of coordinate arrangements.

The \textit{real moment-angle complex} corresponding to a simplicial complex $\K$ on the set $[m]$ is a subcomplex in the cube $I^m = [-1, 1]^m$ 
of the form
$$\R_\K = \bigcup_{\sigma \in \K} (D^1, S^0)^\sigma = \bigcup_{\sigma \in \K}\left({\prod_{i \in \sigma}D^1 \times \prod_{i \notin \sigma} S^0}\right),$$
where $(D^1, S^0)$ is a pair consisting of the interval $[-1, 1]$ and its boundary $\{-1, 1 \}$.

\begin{theorem}[\cite{BP00}, Theorem~5.2.5]
There is a deformation retraction
$$U_\RR(\K) \stackrel{\simeq}\to \R_\K.$$
\end{theorem}

The space $\R_\K$ has the following cellular decomposition. The $i$th factor $[-1, 1]_i$ in the product $I^m=[-1, 1]^m$ is endowed with the natural structure of simplicial complex with vertices
$\underline{t}_i = \{-1\}_i$, $t_i = \{ 1 \}_i$ and with 1-simplex $u_i = [-1, 1]_i$.
Then each cell of the $m$-cube $I^m$ has the form
\begin{equation}\label{basis}
u_\sigma t_\tau \underline{t}_{[m]\backslash (\sigma \cup \tau)} := \prod_{i \in \sigma} u_i \times \prod_{i \in \tau} t_i \times \prod_{i \notin (\sigma \cup \tau)} \underline{t}_i,
\end{equation}
where $\sigma,\tau$ are disjoint subsets of~$[m]$. The cells of the real moment-angle complex $\R_\K\subset I^m$ are distinguished by the condition $\sigma\in\K$.

We identify cells with the corresponding cellular chains. Let us introduce the basis in $C^*(I^m)$ that is used in~\cite{LiCai} for 
the description of the product in~$H^*(\R_\K)$.
Let $\varepsilon_i := \partial u_i = t_i - \underline{t}_i$. Then the cellular chains $u_\sigma \varepsilon_\tau := u_\sigma \varepsilon_\tau \underline{t}_{[m] \backslash (\sigma \cup \tau)}$
form a basis in $C_*(I^m) = \bigotimes_{i=1}^m C_*(I)$. The dual basis consists of cochains of the form
\[
  u^\sigma t^\tau := u^\sigma t^\tau \delta^{[m] \backslash (\sigma \cup \tau)} = \bigotimes_{i \in \sigma} u_i^* \otimes \bigotimes_{i \in \tau} t_i^*\otimes \bigotimes_{i \notin \sigma \cup \tau} \delta_i^*\in
  C^{|\sigma|}(I^m),
\]  
where $u_i^*$, $t_i^*$, $\underline{t}_i^*$ are dual to $u_i$, $t_i$, $\underline{t}_i$, respectively, and $\delta_i^* = t_i^* + \underline{t}_i^*$.
The cellular differential acts on the basis cochains as
\[
  d u_i^*=0,\quad dt_i^*=u_i^*,\quad d\delta_i^*=0.
\]

The product in $C^*(\R_\K)$ is defined as follows. Since $I = [-1, 1]$ is a simplicial complex, there is the standard $\smile$-product, namely:
\begin{gather*}
  t_i^* \smile t_i^* = t_i^*, \quad t_i^* \smile u_i^* = 0, \quad u_i^* \smile t_i^* = 
  u_i^*, \quad u_i^* \smile u_i^* = 0,\\
  \delta_i^* \smile t_i^* = t_i^* \smile \delta_i^* = t_i^*, \quad \delta_i^* \smile u_i^* = u_i^* \smile \delta_i^* = u_i^*, \quad \delta_i^* \smile \delta_i^* = \delta_i^*.
\end{gather*} 
According to the Whitney formula (see~\cite{Whi}), this simplicial product extends to a product in cellular cochains $C^*(I^m)$ by the rule
\[
  u^\sigma t^\tau \smile u^{\sigma'} t^{\tau'} = (-1)^{(\sigma, \sigma')} u^{\sigma \cup \sigma'} t^{\tau \cup (\tau' \setminus \sigma)}
\]
if $\sigma \cap \sigma' = \varnothing$ and $\tau\cap\sigma'=\varnothing$ (otherwise the product is zero), where
\[
  (\sigma, \sigma') = |\{(i, j) \colon i \in\sigma, \; j \in \sigma', \; i > j \}|.
\]  

The product in $C^*(\R_\K)$ is obtained by restricting the product in $C^*(I^m)$ defined above to $\R_\K\subset I^m$.

\begin{theorem}[{\cite[Theorem~5.1]{LiCai}}]
There is an isomorphism of rings
\[ 
  H(C^*(\R_\K),d)\cong H^*(\R_\K).
\]
\end{theorem}

It is proved in~\cite{franz24} that the cellular cochain algebra $C^*(\R_\K)$ with the product above is a dga-model
for $\R_\K$; i.\,e., it is quasi-isomorphic to the singular cochain algebra of~$\R_\K$ with integer coefficients.

We are interested in the dual diagonal of the cellular chain coalgebra $C_*(I^m)$. On the basis chains $u_\sigma t_\tau \underline{t}_{[m] \backslash (\sigma \cup \tau)}$ it is given by
$$\Delta_{C}(u_\sigma t_\tau \underline{t}_{[m] \backslash (\sigma \cup \tau)}) = \sum_{\sigma' \subset \sigma} (-1)^{(\sigma', \sigma \backslash \sigma')} u_{\sigma'} t_\tau \underline{t}_{[m] \backslash (\sigma' \cup \tau)} \otimes u_{\sigma \backslash \sigma'} t_{\sigma' \cup \tau} \underline{t}_{[m] \backslash (\sigma \cup \tau)}.$$
For example, it acts on the top-dimensional cell as
\begin{equation}\label{DeltaC}
  \Delta_{C}(u_{[m]} t_\varnothing \underline{t}_{\varnothing}) = \sum_{\sigma \subset [m]} (-1)^{(\sigma, [m] \backslash \sigma)} u_{\sigma} t_\varnothing \underline{t}_{[m] \backslash \sigma} \otimes u_{[m] \backslash \sigma} t_{\sigma} \underline{t}_{\varnothing}.
\end{equation}

\section{Permutohedral complex $\Perm(\K)$}
The \textit{permutohedron} is the polytope in $\RR^m$ given by
\[
  \Perm^{m-1} = \conv\{\left({ \sigma(1), \ldots, \sigma(m) } \right) \in \RR^m \colon \sigma \in S_m \}.
\]  
There is a well-known description of the face poset of $\Perm^{m-1}$ (see~\cite[Proposition~1.5.5]{BP}, \cite[Chapter~0]{Zieg}).

\begin{theorem}\label{Perm-face}
Faces of $\Perm^{m-1}$ of dimension $p$ are in one-to-one correspondence with ordered partitions of the set $[m]$ into $m - p$ nonempty parts. An inclusion of faces $G \subset F$ holds if and only if the ordered partition corresponding to $G$ can be obtained by refining the ordered partition corresponding to~$F$.
\end{theorem}

We reproduce the proof of this fact, as some of its details will be needed in what follows.

The ordered partition of $[m]$ into nonempty subsets $U_1, \ldots, U_p$ is denoted by $(U_1 | \cdots | U_p)$.
We denote by $F(U_1 | \cdots | U_p )$ the corresponding face of $\Perm^{m-1}$. 
Also, we assume that elements in every part are in increasing ordered.

\begin{proof}[Proof of Theorem~\ref{Perm-face}]
Suppose that a linear functional $\varphi_a(\cdot) = \langle a, \cdot \rangle$
defines a supporting hyperplane of a face $F^p \subset \Perm^{m-1}$;
that is, the restriction of the function to the polytope attains its minimum on~$F$.
We define a partition $(U_1 | U_2 | \cdots | U_{m-p})$ of $[m]$
as follows. If the coordinates of $a$ are ordered as
\begin{equation}\label{covec}
a_{u_{1, 1}} = \cdots = a_{u_{1, i_1}} > a_{u_{2, 1}} = \cdots = a_{u_{2, i_2}} > \cdots > a_{u_{m-p, 1}} = \cdots = a_{u_{m-p, i_{m-p}}},
\end{equation}
then the partition consists of the sets $U_j = \{u_{j, 1}, \ldots, u_{j, i_j} \}$, $j = 1, \ldots, m-p$. 

Consider all vertices $v = (\sigma(1), \ldots, \sigma(m))$ that belong to $F$. 
Since the minimum of $\varphi_a$ on the permutohedron is attained on~$F$,
the sum $\sum_{i = 1}^m a_j v_j$ is minimal among similar sums over all $\sigma \in S_m$.
So, the coordinates of $v$ are ordered as
\begin{equation}\label{nerva}
\{ v_{u_{1, 1}}, \ldots, v_{u_{1, i_1}} \} < \{ v_{u_{2, 1}}, \ldots, v_{u_{2, i_2}} \} < \cdots < \{ v_{u_{m-p, 1}}, \ldots, v_{u_{m-p, i_{m-p}}} \},
\end{equation}
where the notation $A_1 < A_2$ means that each element of $A_1$ is less than each element of~$A_2$.

It follows that if $v$ belongs to $F$, then the corresponding permutation $\sigma$ 
can be decomposed as
\begin{equation}\label{upor}
\sigma = \tau_1 \circ \tau_2 \circ \cdots \circ \tau_{m-p},
\end{equation}
where $\tau_j$ is the bijection between $U_j = \{u_{j, 1}, \ldots, u_{j, i_j} \}$
and
$$\{i_1 + \cdots + i_{j-1} + 1, i_1 + \cdots + i_{j-1} + 2, \ldots, i_1 + \cdots + i_{j-1} + i_j  \}.$$

On the other hand, any vertex that corresponds to the permutation satisfying \eqref{upor} belongs to $F$, 
since the value of $\varphi_a$ at this vertex is the same as at~$v$.

Note that any linear functional $\varphi_b(\cdot) = \langle b, \cdot \rangle$ 
defining the same partition $(U_1 | \cdots | U_{m-p})$
attains its minimum on the same face as~$\varphi_a$.
Indeed, the function $\varphi_b$ attains its minimum on the same vertices $v$ as $\varphi_a$.

Thus, we have proved that there is a bijection between the faces $F$ of the permutohedron and the ordered partitions $(U_1 | \cdots | U_{m-p})$.
Inequalities~\eqref{covec} imply that the dimension of the face $F$ corresponding to a partition $(U_1 | \cdots | U_{m-p})$ is~$p$.
\end{proof}

\begin{proposition}\label{hyp}
Consider the hyperplane 
$$\pi = \Bigl\{ (x_1, \ldots, x_m) \in \RR^m \colon x_1 + \cdots + x_m = \frac{m(m+1)}{2} \Bigr\},$$
that contains $\Perm^{m-1}$. 
There is a deformation retraction $$D_{\RR}(\K) \stackrel{\simeq}\to D_{\RR}(\K) \cap \pi.$$
\end{proposition}
\begin{proof}
The map 
\[
  F(x, t) = x - t (1, \ldots, 1) \Bigl(\frac{x_1 + \cdots + x_m}{m} - \frac{m+1}{2}\Bigr)
\]
defines a homotopy between the identity map $\mathrm{id} \colon \RR^m \to \RR^m$ 
and the projection $\mathrm{pr} \colon \RR^m \to \pi$ to the hyperplane $\pi$. 
There is an inclusion $F(D_{\RR}(\K), t) \subset D_{\RR}(\K)$ for every 
$t \in [0, 1]$, since we subtract the same value from all coordinates of a point $x \in D_\RR(\K)$.
Now $F(D_{\RR}(\K), 1) = D_{\RR}(\K) \cap \pi$, so we obtain the required deformation retraction.
\end{proof}

\begin{lemma}\label{intersection}
Let $F = F(U_1 | \cdots | U_{m-p} )$ be a face of $\Perm^{m-1}$ and $I = \{i_1, \ldots, i_k \}\subset[m]$.
Then $F \cap D_I \neq \varnothing$ if and only if $I \subset U_j$ for some $1 \le j \le p$.
\end{lemma}
\begin{proof}
First we prove the necessity. Suppose that $x \in F \cap D_I$.
Then $x$ is a convex combination of the vertices of $F$, namely:
$x = \sum_{j = 1}^S t_j v_j$, where $v_1, \ldots, v_S$ are all vertices of $F$, $t_j \ge 0$ and $\sum_{j = 1}^S t_j = 1$.
Since each vertex $v_j$ satisfies inequalities \eqref{nerva} and $t_j \ge 0$, it follows that $x$ also satisfies inequalities \eqref{nerva}.  
Therefore, $I$ is contained in some $U_j$.	

Now we prove the sufficiency. Let
\[
  b(F) = \frac1S\sum_{j = 1}^S v_j
\]  
be the barycenter of~$F$.
The decomposition~\eqref{upor} implies that the coordinates of $b(F)$ lying in the same $U_j$ are equal.
Since $I \subset U_j$, we obtain $b(F) \in D_I$, and the intersection $D_I \cap F$ is nonempty.
\end{proof}

\begin{lemma}\label{conv}
Suppose that $I = \{i_1, \ldots, i_k \}$ is a subset in $[m]$. Then 
\[
  D_I \cap \Perm^{m-1} = \conv\{b(F) \in \Perm^{m-1} \colon F \cap D_I \neq \varnothing \}.
\]  
\end{lemma}
\begin{proof}
We let $C = \conv\{b(F) \in \Perm^{m-1} \colon F \cap D_I \neq \varnothing \}$. 
In the proof of Lemma~\ref{intersection} we showed that $b(F)\in D_I$, so $C \subset D_I \cap \Perm^{m-1}$.

We prove the inverse inclusion by contradiction.
Consider a face $F$ of minimal dimension containing a point $x \in D_I \cap \Perm^{m-1}$ such that $x \notin C$.
Since $x \in D_I \cap \Perm^{m-1}$ and $b(F) \in D_I \cap \Perm^{m-1}$, it follows that the interval between $x$ and $b(F)$ belongs to
$D_I \cap \Perm^{m-1}$.  
The line that contains this interval intersects $\partial F$ at some point $y$ that lies in a face of smaller dimension. 
Therefore, $y \in C$. This implies $x \in \conv(\{ b(F)\} \cup C) = C$, which contradicts the choice of $x$.
\end{proof}

\begin{construction}
For each simplicial complex $\K$ on the vertex set $[m]$, we define the \textit{permutohedral complex} $\Perm(\K)$
as the following subcomplex in $\Perm^{m-1}$:
\[
  \Perm(\K) = \Perm^{m-1} \backslash \bigcup_{I \notin \K} \bigcup_{\substack{F = F(U_1 | \cdots | U_p): \\ \exists j\colon 
  I = U_j}} \relint F 
  = \bigcup_{\substack{U_1, \ldots, U_p \in \K, \\ U_1 \sqcup \cdots \sqcup U_p = [m]}} F(U_1 | \cdots | U_p).
\]  
\end{construction}

\begin{theorem}\label{perm-k}
There is a deformation retraction $D_{\RR}(\K) \stackrel{\simeq}\to \Perm(\K)$.
\end{theorem}

\begin{proof}
The deformation retraction $D_{\RR}(\K) \stackrel{\simeq}\to \Perm(\K)$
will be constructed by induction.
We remove simplices from $\Delta^{m-1}$ until we obtain $\K$.

The base of induction is clear: if $\K = \Delta^{m-1}$, then $D_{\RR}(\K) = \RR^m$
and by Proposition~\ref{hyp} we have the retraction $D_{\RR}(\K) \to D_{\RR}(\K) \cap \pi = \pi$ and the evident retraction $\pi \to \Perm^{m-1} = \Perm(\Delta^{m-1})$.

Now assume that $\K$ is obtained from a simplicial complex $\K'$
by removing one maximal simplex $I = \{i_1, \ldots, i_k \} \subset [m]$.
Then the permutohedral complex $\Perm(\K)$ is obtained from $\Perm(\K')$
by removing the relative interiors of faces $F = F(U_1 | \cdots | U_p)$ 
such that $I = U_j$ for some $j$.

Also we have $D_{\RR}(\K) = D_{\RR}(\K') \backslash D_I$.
By the induction hypothesis there is a deformation retraction $D_{\RR}(\K') \to \Perm(\K ')$.
By Lemma~\ref{conv}, 
\[
  D_I \cap \Perm(\K') = \conv\{b(F) \colon F \cap D_I \neq \varnothing \}.
\]  
Consider a face $F \subset \Perm(\K')$ of maximal dimension such that
$F \cap D_I \neq \varnothing$. By the induction hypothesis, $F \cap D_J = \varnothing$ for all $J \notin \K$, $J \neq I$, since all these faces have been removed in the previous step. Hence, $F \backslash \mathcal{DA}(\K) = F \backslash D_I$. 

We define a retraction $r_F \colon F \backslash D_I \to \partial F \backslash D_I$ as follows. 
For any point $x \in F \backslash D_I$ we define $r_F(x)$ to be
the intersection of the ray from $b(F)$ to $x$ with the boundary $\partial F$. 
By Lemma~\ref{conv} the interior of the interval between $b(F)$ and $r_F(x)$ is contained in $F \backslash D_I$, 
so the retraction $r_F\colon F \backslash D_I \to \partial F \backslash D_I$ is well defined and extends identically to the other points of $\Perm(\K') \backslash D_I$.
In case $m = 3$, $I = \{1, 2\}$, and $F = F(123)$, the corresponding retraction is shown in Figure~\ref{retr}.
\begin{figure}
    \centering
    \includegraphics[width=0.4\linewidth]{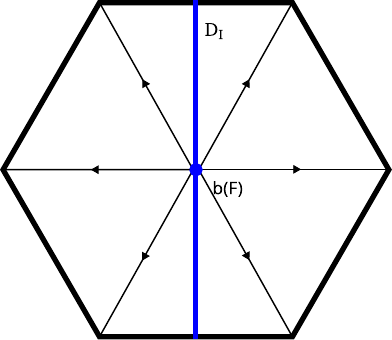}
    \caption{The retraction $r_F$.}
    \label{retr}
\end{figure}

According to Lemma~\ref{conv}, there is a well-defined convex homotopy between the identity map on $\Perm(\K') \backslash D_I$ and the retraction $r_F$.
Hence, we have constructed the deformation retraction $\Perm(\K') \backslash D_I \to (\Perm(\K') \backslash \relint F) \backslash D_I$. The deformation retraction $\Perm(\K') \backslash D_I \to \Perm(\K)$ is constructed similarly, by choosing faces of maximal dimension in $\Perm(\K') \backslash F$ that have nontrivial intersection with $D_I$.
The inductive step is therefore complete.
\end{proof}

The permutohedron and its face subcomplexes arise naturally in many problems involving torus actions. It is well known that the permutohedron is the image of the moment map of the complete flag manifold. Another example is the manifold of isospectral arrow matrices studied in~\cite{AyzB}. It was proved there that the orbit space of the torus action on this manifold is homotopy equivalent to the complex of all cubic faces of the permutohedron. This complex can be described in terms of our construction.

\begin{proposition}
Let $\K$ be the complete graph on $m$ vertices. Then $\Perm(\K)$ is the complex of all cubic faces of the permutohedron, see Figure~\ref{perm-fig}.
\end{proposition}

\begin{proof}
A cubic face of the permutohedron corresponds to a partition $(U_1 | \ldots | U_p)$ in which $|U_j| \le 2$. Since $\K$ is a complete graph, all elements of such a partition are simplices of~$\K$, hence, every cubic face lies in $\Perm(\K)$. Conversely, a face of $\Perm(\K)$ corresponds to a partition $(U_1 | \ldots | U_p)$ into simplices of~$\K$, so the cardinality of each $U_j$ is at most two.
\begin{figure}
    \centering
    \includegraphics[width=0.5\linewidth]{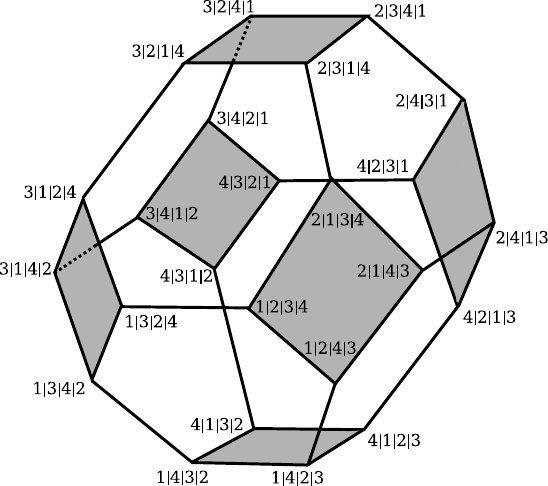}
    \caption{Complex $\Perm(\K)$ for the complete graph on four vertices.}
    \label{perm-fig}
\end{figure}
\end{proof}

We can also obtain a cellular model for the complements of
complex diagonal arrangement using permutohedral complexes. Namely, we consider the space
$\CC^m \cong \RR^{2m}$ and the standard permutohedron there.
Denote the indices $(m+1, \ldots, 2m)$ by $(1', \ldots, m')$. Then the 
faces of $\Perm^{2m-1}$ correspond to ordered partitions of the set
\[
  [m]\cup [m'] = \{1, 2, \ldots, m, 1', 2', \ldots, m'\}.
\]  
Now consider the complex
\[
  \Perm_{\CC}(\K) := \Perm^{2m-1} \backslash \bigcup_{I \notin \K} \bigcup_{\substack{(U_1 | \cdots | U_p) \\ \exists j, k:  I \subset U_j, I' \subset U_k}} \relint F(U_1 | \cdots | U_p),
\]  
where each set $I' = \{i_1', \ldots, i_s' \}$ of primed indices corresponds to the set $I = \{i_1, \ldots, i_s\}$ of the same indices without primes.

\begin{theorem}
There is a homotopy equivalence
$$D_\CC(\K) \simeq \Perm_\CC(\K).$$
\end{theorem}
\begin{proof}
As in the proof of Theorem~\ref{perm-k}, we use induction to remove faces from $\Delta^{m-1}$ until we obtain the simplicial complex $\K$.
The difference with the real case is that we remove the relative interiors of faces $F$ such that $F \cap D^\CC_I \neq \varnothing$; that is,
we remove the relative interiors of faces $F$ that intersect both $D_I$ and $D_{I'}$. 
By Lemma~\ref{intersection} these faces $F$ correspond to partitions
$(U_1| \cdots | U_p)$ with $I \subset U_j$, $I' \subset U_k$ for some $j$ and $k$.
The barycenter of such a face is contained in the intersection $F \cap D^\CC_I$,
so we define the retraction $r_F \colon F \backslash D^\CC_I \to \partial F \backslash D^\CC_I$ similarly to the real case, by constructing rays from $b(F)$.
The remaining part of the proof is completely the same as for Theorem~\ref{perm-k}.
\end{proof}

\section{Algebraic model for cellular cochains}
Let $\kk$ be a commutative ring with unit and $A$ be a graded $\kk$-algebra with unit.
We assume that $A$ is connected, that is, $A_{<0}=0$ and $A_0 = \kk \cdot 1$, so we have the augmentation map $\varepsilon \colon A \to \kk \cong A_0$.
We denote the augmentation ideal by $I(A) = \ker \varepsilon$. 
The \textit{bar construction} of the algebra $A$ is a chain complex of the form
\[
  \cdots \to B^{-n}(A) \stackrel{d}\to B^{-n+1}(A) \to \cdots \to B^{-1}(A) \stackrel{d}\to B^0(A) \stackrel{\varepsilon}\to \kk \to 0,
\]  
where $B^{-n}(A) = A \otimes_{\kk} \underbrace{I(A) \otimes_{\kk} \cdots \otimes_{\kk} I(A)}_n$.

Elements $a \otimes a_1 \otimes \cdots \otimes a_n \in B^{-n}(A)$ are traditionally denoted by $a[a_1 | \cdots | a_n]$.
The differential is given by
$$-d(a[a_1 | \cdots | a_n]) = \overline{a} a_1 [a_2 | \cdots | a_n] + \sum_{i = 1}^{n-1} \overline{a}[\overline{a}_1 | \cdots | \overline{a}_{i-1} |
\overline{a}_i a_{i+1} | a_{i+2} | \cdots | a_n],$$
where $\overline{a}_j = (-1)^{\deg a_j + 1}a_j$. 
It is well known that the bar construction is a free resolution of the left $A$-module $\kk$ (see~\cite[Chapter~X]{Mac}).

After applying the functor $(\kk \otimes_A - )$ to the complex $(B(A), d)$, we obtain the chain complex $(\overline B (A) , \overline d )$ called
the \textit{reduced bar construction} of the ring $\kk$ as a left $A$-module. 
Elements of the reduced bar construction are $[a_1 | \cdots | a_n] \in \overline{B}^{-n}(A)$, where $a_j \in I(A)$. 
The differential $\overline d$ is defined by
\[
  -\overline d ([a_1 | \cdots | a_n]) = \sum_{i = 1}^{n-1} [\overline{a}_1 | \cdots | \overline{a}_{i-1} | \overline{a}_i a_{i+1} | a_{i+2} | \cdots | a_n].
\] 
The $(-n)$th cohomology of this complex is the $\Tor$ module:
\[
  \Tor_A^{-n}(\kk, \kk) = H^{-n}[\overline B (A), \overline d].
\]  
 
Now let $A$ be the exterior Stanley--Reisner algebra of a simplicial complex $\K$. It is the quotient algebra
\[
  \Lambda[\K] = \Lambda[x_1, \ldots, x_m] / \mathcal{I}_{SR},
\]  
where $\Lambda[x_1, \ldots, x_m]$ is the exterior algebra with $m$ generators, $\deg x_i = 1$, $i = 1, \ldots, m$, and
$\mathcal{I}_{SR} = (x_{i_1} \cdots x_{i_k}\colon \{i_1, \ldots, i_k \} \notin \K)$ %\triangleleft \Lambda[x_1, \ldots, x_m]$
is the Stanley--Reisner ideal.
There is a $\ZZ_{\ge 0}^m$-grading in the exterior Stanley--Reisner algebra, defined as follows. If $X = x_{i_1} \cdots x_{i_k}$,
then $\mdeg(X) = ( q_1, \ldots, q_m )$, where $q_j = 1$ if $j \in \{i_1, \ldots, i_k \}$ and $q_j = 0$ otherwise.

Consider the reduced bar construction $\overline{B}(\Lambda[\K])$. 
The basis of $\overline{B}^{-n}(\Lambda[\K])$ consists of the elements $[X_1 | \cdots | X_n]$,
where $X_i$ is a monomial in $\Lambda[\K]$, $i = 1, \ldots, n$. 
Hence, there is a $\ZZ_{\ge 0}^m$-grading in the reduced bar construction defined by $\mdeg([X_1 | \cdots | X_n]) = \mdeg(X_1) + \cdots + \mdeg(X_n)$.

We see that the differential $\overline{d}$ preserves the multigrading.
So, the reduced bar construction $(\overline{B}(\Lambda[\K]), \overline d)$ splits into a direct sum of multigraded components $(\overline{B}(\Lambda[\K]), \overline d)_{ (\alpha_1, \ldots, \alpha_m ) }$,
and the $\kk$-modules
$\Tor_{\Lambda[\K]}^*(\kk, \kk)_{(\alpha_1, \ldots, \alpha_m )}$ are defined. 
We are interested in the component $(\overline{B}(\Lambda[\K]),\overline{d})_{(1, \ldots, 1 )}$.

\begin{theorem}\label{alg-cell}
The cellular cochain complex of $\Perm(\K)$ is isomorphic to the $(1 , \ldots, 1 )$-component of the reduced bar construction of the exterior Stanley--Reisner algebra:
\[
  \bigl(C^p(\Perm(\K); \kk), d\bigr) \cong \bigl(\overline{B}^{p-m}(\Lambda[\K]),\overline{d}\bigr)_{(1, \ldots, 1 )}.
\]  
\end{theorem}
\begin{proof}
Consider the $\kk$-module homomorphism 
\[
  \varphi\colon C^*(\Perm(\K)) \to 
  \overline{B}(\Lambda[\K])_{(1, \ldots, 1 )},
\]
defined on the generators by
\[
  \varphi (F(U_1 | \cdots | U_{m-p})^*) = [X_1 | \cdots | X_{m-p}],
\]  
where $F(U_1 | \cdots | U_{m-p})^*$ is the cochain dual to a $p$-dimensional face $F^p = F(U_1 | \cdots | U_{m-p})$, and $X_j = \prod_{i \in U_j} x_i$.
The map $\varphi$ is an isomorphism of $\kk$-modules, since it is bijective on generators (note that $\overline{B}^{-n}(\Lambda[\K])_{(1, \ldots, 1 )} = 0$ if $n > m$).

The boundary map in the cellular chain complex of permutohedron is calculated in~\cite{Mil}. It has the form
\begin{multline*}
  \partial F(U_1 | \cdots | U_{m-p}) = \\
  = \sum_{j = 1}^{m-p} \sum_{M \subset U_j} 
  (-1)^{\varepsilon} \mathop{\mathrm{shuff}} (M; U_j \backslash M) 
  F(U_1 | \cdots | U_{j-1} |  M | U_j \backslash M | U_{j+1} | \cdots | U_{m-p}),
\end{multline*}
where $\varepsilon = {m_1 + \cdots + m_{j-1} + |M|}$, $m_i = |U_i| - 1$, and $\mathop{\mathrm{shuff}}(M; U_j \backslash M)$ is the sign of the permutation sending the first $|M|$ elements of $U_j$ to the elements of $M$,
and the other elements to $U_j \backslash M$.

After dualizing we find that the differential acts on cochains as
\[
  d F(U_1 | \cdots | U_{m-p})^* = \sum_{j = 1}^{m-p} (-1)^\varepsilon \mathop{\mathrm{shuff}}(U_j; U_{j+1}) F(U_1 | \cdots | U_{j-1} | V_j | U_{j+2} | \cdots | U_{m-p})^*,
\]  
where $V_j = U_j \cup U_{j+1}$ and we assume that the elements of this set are in increasing order. 
Applying the homomorphism $\varphi$, we obtain
\[
  \varphi (dF(U_1 | \cdots | U_{m-p})^*) = \sum_{j = 1}^{m-p} (-1)^\varepsilon \mathop{\mathrm{shuff}}(U_j; U_{j+1}) [X_1 | \cdots | (-1)^{\delta_j} X_j X_{j+1} | \cdots | X_{m-p}],
\]
where the sign $(-1)^{\delta_j}$ is defined by the identity $(-1)^{\delta_j} X_j X_{j+1} = \prod_{k \in V_j} x_k$. 
Note that $\delta_j = \mathop{\mathrm{shuff}}(U_j; U_{j+1})$ and $m_i = \deg X_i - 1$, so $\varphi( dF(U_1 | \cdots | U_{m-p})^*) = d (\varphi(F(U_1 | \cdots | U_{m-p})^*))$.
Thus, we have shown that $\varphi$ is the required isomorphism of chain complexes.
\end{proof}

\begin{corollary}
For any commutative ring $\kk$ with unit there is an isomorphism of $\kk$-modules
$$H^p(D_\RR(\K); \kk) \cong \Tor^{p-m}_{\Lambda[\K]}(\kk, \kk)_{ ( 1, \ldots, 1) }.$$
\end{corollary}

\begin{remark}
This isomorphism was also proved by Dobrinskaya in~\cite{Dobr09} using the CW complex with cells homeomorphic to products of permutohedra. 
However, this cell decomposition is constructed less explicitly using the theory of monoidal completions developed in~\cite{Dobr06}.
Our construction allows us to describe the product in the cohomology ring using cellular diagonal approximations in the permutohedron.
\end{remark}

\section{Saneblidze--Umble diagonal}
Since $\Lambda[\K]$ is a graded commutative algebra,
there is a natural graded commutative product in the $\kk$-module $\Tor_{\Lambda[\K]}^*(\kk, \kk)$ that arises from the bar construction (see~\cite[Chapter~X]{Mac}).
However, this product does not preserve the grading $(1, \ldots, 1 )$, 
so it cannot be used for description of the product in the cohomology algebra $H^*(\Perm(\K); \kk)$, which is additively isomorphic to $\Tor_{\Lambda[\K]}^*(\kk, \kk)_{(1, \ldots, 1 )}$. 

A product of cellular cochains $C^*(\Perm(\K))$ can be defined using a cellular diagonal approximation $\widetilde{\Delta} \colon \Perm^{m-1} \to \Perm^{m-1} \times \Perm^{m-1}$ of the standard diagonal in the permutohedron.
One of these cellular diagonal approximations is constructed in~\cite{SU}.
Here we give the definition of this approximation.

A $q \times p$ matrix $O = (o_{i,j})$ is called \textit{ordered} if the following conditions hold:
\begin{enumerate}
\item[(1)] $\{o_{i,j}\} = \{0, 1, \ldots, q+p-1 \}$;
\item[(2)] each row and column of $O$ is non-zero;
\item[(3)] non-zero entries in $O$ are distinct and increase in each row and column.
\end{enumerate}
The set of ordered $q\times p$ matrices is denoted by $\mathcal{O}^{q \times p}$.

Given an ordered matrix $O$, we consider two partitions of $[q+p-1]$. The first partiiton is $c(O) = (O_1 | \cdots | O_q)$,
where $O_j$ is the $j$th column of $O$, and the second is $r(O) = (O^p | \cdots | O^1)$, where $O^i$ is the $i$th row of $O$ (note the reverse order of rows).
Here we assume that all zero entries of $O_j$ and $O^i$ are removed.

An ordered matrix $E=(e_{i,j})$ is a \textit{step matrix} if
\begin{enumerate}
\item[(1)] the nonzero entries in each row (column) of $E$ appear in consecutive columns (respectively, rows);
\item[(2)] the main diagonal as well as each secondary diagonal of $E$ contain a single nonzero entry (i.\,e., there is a single nonzero entry among $e_{i,j}$ with fixed $j-i$).
\end{enumerate}
The set of $q \times p$ step matrices is denoted by $\mathcal{E}^{q \times p}$.

\begin{example}\label{example1}
The matrix $E = \begin{pmatrix}
0 & 2 & 3 \\
1 & 5 & 0 \\
4 & 0 & 0
\end{pmatrix}$ is a step matrix. The corresponding partitions of $\{1,2,3,4,5\}$ are $c(E) = (14| 25| 3)$ and $r(E) = (4| 15| 23)$.
\end{example}

For any $(i, j) \in \ZZ^+ \times \ZZ^+$ define the \textit{down-shift} operators $D_{i,j}\colon \mathcal{O} \to \mathcal{O}$ and the \textit{right-shift} operators
$R_{i,j} \colon \mathcal{O} \to \mathcal{O}$
on $O \in \mathcal{O}^{q \times p}$ as follows:
\begin{enumerate}
\item[1)] If $o_{i,j}>0$, $o_{i+1, j} = 0$, $o_{i+1, l} < o_{i,j}$ whenever $l < j$, $o_{i+1, l} > o_{i,j}$ whenever $l > j$ and $o_{i+1, l} \neq 0$,
and there is $o_{i, k} \neq 0$ for some $k \neq i$, then
$D_{i,j}O$ is obtained from $O$ by transposing $o_{i, j}$ and $o_{i+1, j}$.
Otherwise $D_{i,j} O = O$.

\item[2)] If $o_{i,j}>0$, $o_{i, j+1} = 0$, $o_{l, j+1} < o_{i,j}$ whenever $l < i$, $o_{l, j+1} > o_{i,j}$ whenever $l > i$ and $o_{l, j+1} \neq 0$,
and there is $o_{k, j} \neq 0$ for some $k \neq j$, then
$R_{i,j}O$ is obtained from $O$ by transposing $o_{i, j}$ and $o_{i, j+1}$.
Otherwise $R_{i,j} O = O$.
\end{enumerate}

\begin{example}\label{example2}
Let $E = \begin{pmatrix}
0 & 2 & 3 \\
1 & 5 & 0 \\
4 & 0 & 0
\end{pmatrix}$. Then 
\[
 R_{2,2}E = \begin{pmatrix} 
0 & 2 & 3\\
1 & 0 & 5\\
4 & 0 & 0
\end{pmatrix},\quad 
D_{2,2}E = \begin{pmatrix} 
0 & 2 & 3\\
1 & 0 & 0\\
4 & 5 & 0
\end{pmatrix},\quad 
D_{2,3} R_{2,2}E = \begin{pmatrix} 
0 & 2 & 3\\
1 & 0 & 0\\
4 & 0 & 5
\end{pmatrix}.
\]
\end{example}

A matrix $A \in \mathcal{O}$ is a \textit{configuration matrix} if there exists a step matrix $E$ and a sequence of shift operators 
$G_1, \ldots, G_s$ such that
\begin{enumerate}
\item[1)] $A = G_s \cdots G_1 E$,
\item[2)] if $G_s \cdots G_1 = \cdots D_{i_2, j_2} \cdots D_{i_1, j_1} \cdots$, then $i_1 \le i_2$,
\item[3)] if $G_s \cdots G_1 = \cdots R_{i_2, j_2} \cdots R_{i_1, j_1} \cdots$, then $j_1 \le j_2$.
\end{enumerate}
When this occurs, we say that $A$ is derived from $E$.
The set of $q \times p$ configuration matrices is denoted by $\mathcal{C}^{q \times p}$.

\begin{remark}
All matrices of Example~\ref{example2} are configuration matrices.
\end{remark}

For each $m$, define
\begin{equation}\label{SU-diag}
\Delta_{SU}(F(1, 2, \ldots, m)) = \sum_{q = 1}^m \sum_{A \in \mathcal{C}^{q \times (m-q+1)}} \mathop{\mathrm{csgn}}(A)\; F(c(A)) \otimes F(r(A))
\end{equation}
and extend $\Delta_{SU}$ to the other faces $F(U_1 | \cdots | U_p)$ via the standard comultiplicative extension; that is, 
\begin{equation}\label{comulrule}
  \Delta_{SU}\bigl(F(U_1| \cdots | U_p)\bigr) = 
  F\bigl(\Delta_{SU}(F(U_1)) \big| \cdots \big| \Delta_{SU}(F(U_p))\bigl).
\end{equation}
The sign $\mathop\mathrm{csgn}(A)$ is defined in the following way.
If $A \in \mathcal{C}^{q \times p}$ is derived from a step matrix $E$, then
\[
  \mathop{\mathrm{csgn}}(A) = (-1)^{\binom q2} 
  \mathop{\mathrm{rsgn}}(c(E)) \cdot \mathop{\mathrm{sgn}}\nolimits_1 r(A) \cdot \mathop{\mathrm{sgn}}\nolimits_2 c(E) \cdot \mathop{\mathrm{sgn}}\nolimits_2 c(A),
\]  
where
\begin{align*}
  \mathop{\mathrm{rsgn}}(U_1 | \cdots |U_p) &= (-1)^{\frac12(|U_1|^2 + \cdots + |U_p|^2 - m)},\\
  \mathop{\mathrm{sgn}}\nolimits_1(U_1 | \cdots | U_p) &= (-1)^{\sum_{i=1}^{p-1}i \cdot |U_{p-i}|} 
  \mathop{\mathrm{psgn}}(U_1 \cup \cdots \cup U_p),\\ 
  \mathop{\mathrm{sgn}}\nolimits_2(U_1 | \cdots | U_p) &= 
  (-1)^{{\binom{p-1}2} + \sum_{i=1}^{p-1}i \cdot |U_{p-i}|} 
  \mathop\mathrm{{psgn}}(U_1 \cup \cdots \cup U_p),
\end{align*}
$\mathop\mathrm{{psgn}}(U_1 \cup \cdots \cup U_p)$ is the sign of the permutation of $[m]$ that sends the first $|U_1|$ elements to $U_1$, the next $|U_2|$ elements to $U_2$, and so on, ending with the last $|U_p|$ elements sent to~$U_p$.

Note that the diagonal~\eqref{SU-diag} is constructed combinatorially and it is not proved in~\cite{SU} that~\eqref{SU-diag} is a cellular diagonal approximation.
However, it was later shown in~\cite{DLPS} that \eqref{SU-diag} is indeed homotopic to the standard diagonal.

\begin{example}
Let $m = 2$. Then any ordered $1\times 2$ or $2 \times 1$ matrix is a step matrix. Therefore,
$$\Delta(F(12)) = F(12) \otimes F(2|1) + F(1|2) \otimes F(12).$$
In the case $m = 3$ the formula for $\Delta(F(123))$ is obtained in~\cite{SU}: 
\begin{multline*}
  \Delta(F(123)) = F(1|2|3) \otimes F(123) + F(123) \otimes F(3|2|1) -\\ -F(1|23) \otimes F(13|2) + F(2|13) \otimes F(23|1) - F(13|2) \otimes F(3|12) +\\ 
  + F(12|3) \otimes F(2|13) - F(1|23) \otimes F(3|12) + F(12|3) \otimes F(23|1).
\end{multline*}  
\end{example}

\section{The connection between the Cai and Saneblidze--Umble diagonals}
We use the piecewise linear projection $\rho \colon \Perm^{m-1} \to I^{m-1}$ defined in~\cite{SU} to describe the connection between permutohedral and real moment-angle complexes.
Its action on the vertices of permutohedron is given by
$$\rho( F(U_1 | \cdots | U_m)) = \prod_{i \in \tau} \{ 1\} \times \prod_{i \notin \tau} \{-1\},$$
where $\tau = \{i \colon U_j=\{i+1\},\; U_k=\{i\} \text{ for some }j < k\}$.
Since this map is piecewise linear, we can recover the image of each face from the image of its vertices, namely,
\begin{equation}\label{rho-face}
\rho \left({F(U_1 | \cdots | U_p) }\right) = \prod_{i \in \sigma}D^1 \times \prod_{i \in \tau} \{ 1 \} \times \prod_{i \notin \sigma \cup \tau} \{-1\},
\end{equation}
where  $\sigma = \{i \mid \exists j : \{i, i+1 \} \subset U_j \}$, 
$\tau = \{i \mid \exists j < k:  i + 1 \in U_j, \text{ } i \in U_k\}.$

\begin{lemma}
Consider a face $F = F(U_1| \cdots | U_p)$ of $\Perm^{m-1}$.
The dimension of the face $\rho(F)$ is equal to the dimension of $F$ if and only if
each subset $U_j$ has the form
$$U_j = \{a_j, a_j+1, \ldots, a_j + |U_j| - 1 \}, \qquad a_j \in [m].$$
\end{lemma}
\begin{proof} 
From the description of faces of the permutohedron, we see that $\dim F = m - p$.
Relation~\eqref{rho-face} implies that $\dim \rho(F) = |\sigma|$.
Note that for each $j = 1, 2, \ldots, p$ the element $\max(U_j)$ does not belong to $\sigma$, that is, $|\sigma| \le  m-p$.
So, the equality $\dim \rho(F) = \dim F$ holds if and only if all other elements of $[m]$ belong to $\sigma$.
This is equivalent to the description of $U_j$ given in the statement of the lemma.
\end{proof}

\begin{definition}
We say that the elements $i, i+1, \ldots, i+k$ of an ordered matrix
$O$ form a \textit{snake} if there are numbers $0 = k_0 < k_1 <  \cdots < k_r = k$ such that $i + k_j, i+k_j + 1, \ldots, i+k_{j+1}$ lie in the same row for even $j$ and in the same column for odd $j$, or lie in the same column for even $j$ and in the same row for odd $j$.
The elements $i + k_0, \ldots, i + k_r$ are called the \textit{nodes} of the snake.

A snake in an ordered matrix $O$ is called \textit{continuous}
if all its elements lying in the same row (in the same column) appear in consecutive columns (respectively, in consecutive rows).
\end{definition}

\begin{example} In the matrix
$$O = \begin{pmatrix}
1 & 0 & 2 & 3 & 0 & 0 & 0\\
0 & 0 & 0 & 4 & 0 & 0 & 0\\
0 & 0 & 0 & 5 & 6 & 7 & 0\\
0 & 9 & 0 & 0 & 0 & 0 & 10\\
0 & 0 & 0 & 0 & 0 & 8 & 11
\end{pmatrix},$$
the elements $\{1, 2, 3, \ldots, 8 \}$ and $\{9, 10, 11\}$ form snakes. 
The snakes formed by the elements $\{2, 3, \ldots, 7 \}$ and $\{10, 11\}$ are continuous.
\end{example}

\begin{lemma}\label{snake}
Suppose that a matrix $A \in \mathcal{O}^{q \times (m-q+1)}$ 
is such that the dimension of the faces 
$F(c(A))$ and $F(r(A))$ is preserved under the projection $\rho$.
Then all nonzero entries of $A$ form one continuous snake. 

\end{lemma}
\begin{proof} Since the dimension of the faces $F(c(A))$ and $F(r(A))$ is preserved under the projection $\rho$, the nonzero entries in each row and column of $A$ are consecutive numbers. 

For each $i \in [m]$ we construct the maximal snake $S_i$ containing $i$ as follows. 
Let $A_j$ be some column of $A$ and let $\{i - l_1, i-l_1+1, \ldots, i+k_1\}$ be all its nonzero entries ($l_1$ or $k_1$ may be 0).
For each $-l_1 < j < k_1$ the element $i + j$ is the unique nonzero entry in its row.
Indeed, all nonzero entries of every row are consecutive numbers, but both $i+j-1$ and $i+j+1$ are in the same column~$A_j$.

For the same reason, the row of the element $i-l_1$ does not contain nonzero entries on the right of it.
We denote the nonzero entries of this row by $i - l_2, i - l_2 + 1, \ldots, i-l_1$.
Then the columns containing the entries $i - l_2 + 1, \ldots, i-l_1 - 1$
do not contain other nonzero entries, and the column containing $i-l_2$ does not contain nonzero entries below $i-l_2$.
Hence, $\{ i-l_2, \ldots, i-l_1, \ldots, i+k_1\}$ is a snake.
Now we extend this snake step by step by adding rows or columns containing the first or the last element of the snake. 
At some step, we get the entries $i-l_a$ and $i+k_b$ that are the only nonzero entries in their row or column not included in the snake. 
As a result, we obtain a maximal snake~$S_i$.

We mark all rows and columns of $A$ that contain at least one element of $S_i$.
At the intersection of such a row and column, we have either zero or an element of the snake.
This implies that the total number of rows and columns occupied by the snake $S_i$ is the number of elements in $S_i$ plus one.

It remains to notice that each $i \in [m]$ belongs to exactly one maximal snake,
so the total number of elements of all snakes is equal to $m$.
Since there are no zero rows and columns in $A$, each row and column of $A$ corresponds to exactly one snake.
Hence, the total number of rows and columns of all snakes is $q + (m-q + 1) = m+1$. It follows that there is only one maximal snake, and it contains all nonzero entries of $A$. 
This snake is continuous since all rows and columns of $A$ are nonzero.
\end{proof}

\begin{theorem}
For any face $F(U_1 | \cdots | U_p)$ of $\Perm^{m-1}$ we have
$$(\rho_* \otimes \rho_*) \Delta_{SU} F(U_1 | \cdots | U_p) = \Delta_{C}(\rho_* F(U_1 | \cdots | U_p)).$$
\end{theorem}
\begin{proof}
By the comultiplicative rule~\eqref{comulrule} it is sufficient to consider the action of the homomorphism on the face of maximal dimension.
Consider $\Delta_{SU}F([m])$. By Lemma~\ref{snake}, the image of $\rho_* \colon C_{m-1}(\Perm^{m-1}) \to C_{m-1}(I^{m-1})$ consists of those
summands in~\eqref{SU-diag} that correspond to configuration matrices with one continuous snake.

We claim that all matrices satisfying the hypothesis of Lemma~\ref{snake} are configuration matrices.
Indeed, every ordered matrix $A$ containing one continuous snake with nodes $1=i_0, i_1, \ldots, i_r = m$ can be derived from a step matrix $E$ of the following form. 
Suppose that the elements $\{1, 2, \ldots, i_1 \}$ form a row and $r$ is even.
The first row of $E$ is formed by entries $\{1, 2, \ldots, i_1, i_2+1, i_2 + 2, \ldots, i_3, i_4+1, i_4 + 2, \ldots, i_{r - 1}\}$ arranged in consecutive columns, and the first column of $E$ is formed by entries
$\{1, i_1 + 1, i_1 + 2, \ldots, i_2, i_3 + 1, i_3 + 2, \ldots, i_r \}$ arranged in consecutive rows. (If $r$ is odd, we should swap $i_{r-1}$ and $i_r$, and if entries $\{1, 2, \ldots, i_1 \}$ form a column, we should transpose the matrix.)

At the $j$th step, we will shift the column of elements $\{i_j + 1, \ldots, i_r \}$ to the right until we reach the column containing the element $i_j$. Then at the $(j+1)$th step, we shift the row of elements $\{i_{j+1} + 1, \ldots, i_{r-1} \}$ down until the row containing element the $i_{j+1}$. The procedure is then performed until we obtain the matrix $A$. The first two steps of this procedure are as follows:
$$\begin{pmatrix}
1 & 2 & \cdots & i_1 & i_2 + 1 & \cdots & i_{r-1} \\
i_1 + 1 & 0 & \cdots & 0 & 0 & \cdots & 0 \\
i_1 + 2 & 0 & \cdots & 0 & 0 & \cdots & 0 \\
\vdots & \vdots & \cdots & \vdots & \vdots & \cdots & \vdots\\
i_2 & 0 & \cdots & 0 & 0 & \cdots & 0 \\
\vdots & \vdots & \cdots & \vdots & \vdots & \cdots & \vdots \\
i_{r} & 0 & \cdots & 0 & 0 & \cdots & 0 
\end{pmatrix} \to
\begin{pmatrix}
1 & 2 & \cdots & i_1 & 0 & \cdots & 0 \\
0 & 0 & \cdots & i_1 + 1 & 0 & \cdots & 0 \\
0 & 0 & \cdots & i_1 + 2 & 0 & \cdots & 0 \\
\vdots & \vdots & \cdots & \vdots & \vdots & \cdots & \vdots\\
0 & 0 & \cdots & i_2 & i_2 + 1 & \cdots & i_{r-1} \\
\vdots & \vdots & \cdots & \vdots & \vdots & \cdots & \vdots \\
0 & 0 & \cdots & i_{r} & 0 & \cdots & 0 
\end{pmatrix}$$

Let $\sigma = \{i_1, \ldots, i_2 - 1, i_3, \ldots, i_4 - 1, \ldots, i_{r} - 1\}$ and $\sigma' = \{1, \ldots, i_1-1, i_2, \ldots, i_3-1, \ldots, i_{r} \}$. Relation~\eqref{rho-face} implies that $\rho(F(c(A))) = u_{\sigma} t_\varnothing \underline{t}_{\sigma'}$ and
$\rho(F(r(A))) = u_{\sigma'} t_{\sigma} \underline{t}_\varnothing$.
Since each subset $\sigma \subset [m]$ corresponds to a single snake, we have
\[
  (\rho_* \otimes \rho_*)\Delta_{SU}F([m]) = \sum_{\sigma \subset [m]} 
  (-1)^\varepsilon u_\sigma t_\varnothing \underline{t}_{\sigma'} \otimes 
  u_{\sigma'}t_{\sigma} \underline{t}_\varnothing.
\]
Direct calculations show that the sign $(-1)^\varepsilon$ coincides with the sign in formula~\eqref{DeltaC} for~$\Delta_{C}$.
\end{proof}

\begin{theorem}\label{real-complex} Let $\K$ be a simplicial complex on the vertex set $[m]$. Then the image
of $\Perm(\K)$ under the projection $\rho \colon \Perm^{m-1} \to I^{m-1}$ is the real moment-angle complex $\R_\mathcal{L}$
for some simplicial complex $\mathcal{L} = \mathcal{L}(\K)$ on the set $[m-1]$.
\end{theorem}

\begin{proof}
For each $J \subset [m-1]$ we consider the preimage of $(D^1, S^0)^J$ under the map $\rho \colon \Perm^{m-1} \to I^{m-1}$.
We can write $J$ as 
\[
  J = \{j_1, j_1 + 1, \ldots, j_1+k_1, j_2, j_2 + 1, \ldots, j_2 + k_2, \ldots, 
  j_s + k_s \},
\]   
where $j_{i+1} > j_i + k_i + 1$.

Relation~\eqref{rho-face} implies that the projection of a face $F(U_1 | \cdots | U_p)$ is a $|J|$-dimensional face of $(D^1, S^0)^J$ if and only if each set
$V_l = \{j_l, j_l + 1, \ldots, j_l + k_l + 1 \}$ is contained in some $U_i$.
If a face $F(U_1 | \cdots | U_p)$ with such a property belongs to the complex $\Perm(\K)$, then each $V_l$ is contained in~$\K$.
On the other hand, suppose that each $V_l$ belongs to~$\K$. 
Consider the faces of permutohedron that correspond to the partitions of $[m]$ into sets $\{ V_1, \ldots, V_s \} \cup ([m] \backslash \bigcup_i V_i)$ in some order. 
Then the union of images of all these faces is $(D^1, S^0)^J$.

To complete the proof, it remains to note that the sets $J$ that satisfy $(D^1, S^0)^J \subset \rho(\Perm(\K))$ indeed form a simplicial complex.
\end{proof}

The proof of Theorem~\ref{real-complex} provides the following description of the simplicial complex $\mathcal{L}(\K)$. 
A set $J \subset [m-1]$ belongs to $\mathcal{L}(\K)$ if and only if each subset formed by consecutive elements $\{j, j+1, \ldots, j+k \}$ of $J$
together with the element $\{j + k + 1 \}$ forms a simplex of~$\K$.

\begin{example}
\begin{figure}
    \centering
    \includegraphics[width=0.75\linewidth]{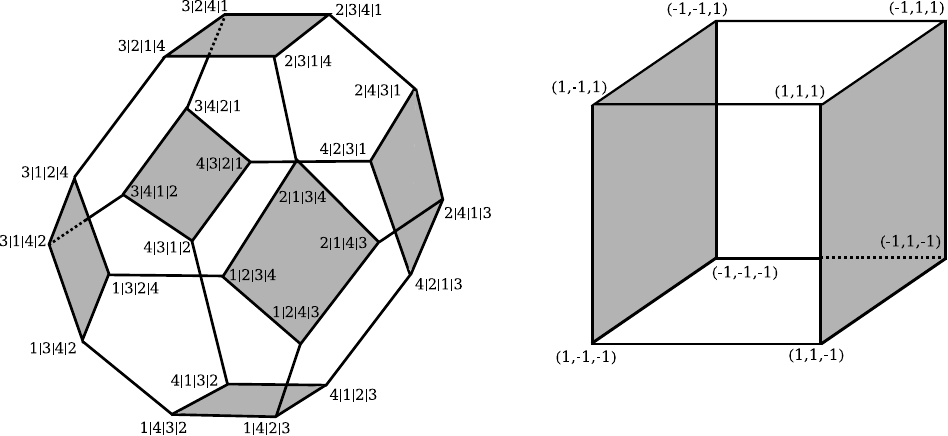}
    \caption{Permutohedral complex $\Perm(\K_1)$ and moment-angle 
    complex $\R_{\mathcal{L}_1}$.}
    \label{perm-k4}
\end{figure}
1. Let $\K_1 = \sk^1 \Delta^3$ be the complete graph on 4 vertices as in Figure~\ref{perm-fig}. Then $\mathcal L(\K_1)$ is the simplicial complex $\mathcal L_1 = \{ \{1\}, \{2\}, \{3\}, \{1, 3\}\}$ on the set $[3] = \{1, 2, 3\}$. 
The real moment-angle complex $\mathcal{L}_1$ is shown in Figure~\ref{perm-k4} on the right.
The projection $\rho \colon \Perm(\K_1) \to \R_{\mathcal{L}_1}$ maps the  front and rear faces $F(12|34)$ and $F(34|12)$ homeomorphically to
the two-dimensional faces of cube that lie in the planes $\{y = -1\}$ and $\{y = 1 \}$, respectively.
The top and bottom faces $F(23|14)$ and $F(14|23)$ are mapped to the edges on the lines $\{x = 1, z = -1 \}$ and $\{x=-1, z=1 \}$, respectively.
The left and right faces $F(13|24)$ and $F(24|13)$ are mapped to the vertices
$(1, -1, 1)$ and $(-1, 1, -1)$, respectively.
\begin{figure}
    \centering
    \includegraphics[width=0.75\linewidth]{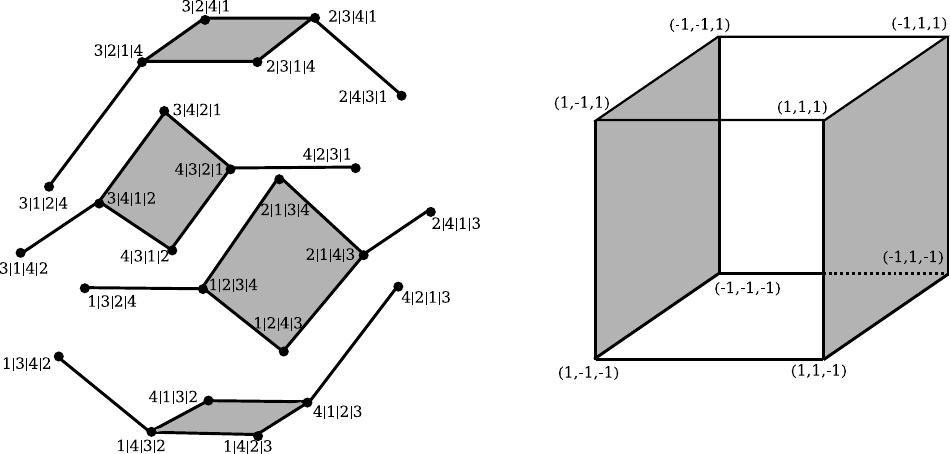}
    \caption{Permutohedral complex $\Perm(\K_2)$ and moment-angle 
    complex $\R_{\mathcal{L}_2}$.}
    \label{perm-1234}
\end{figure}

\smallskip

2. Now let $\K_2$=\begin{picture}(37, 20)
\multiput(10, 2)(15, 0){2}{\line(0, 1){15}}%
\multiput(10, 2)(0, 15){2}{\line(1, 0){15}}%
\put(10, 2){\circle*{3}} \put(25, 2){\circle*{3}} \put(10, 17){\circle*{3}} \put(25, 17){\circle*{3}}%
\put(2, 0){4} \put(30, 0){3} \put(2, 13){1} \put(30, 13){2}
\end{picture} be the boundary of a quadrilateral.
The corresponding permutohedral complex $\Perm(\K_2)$ and its image $\R_{\mathcal L_2}$ under $\rho$ are shown in Figure~\ref{perm-1234}.
We see that the simplicial complex $\mathcal{L}_2$ is the same as the complex $\mathcal L_1$ from the previous example.
\begin{figure}
    \centering
    \includegraphics[width=0.75\linewidth]{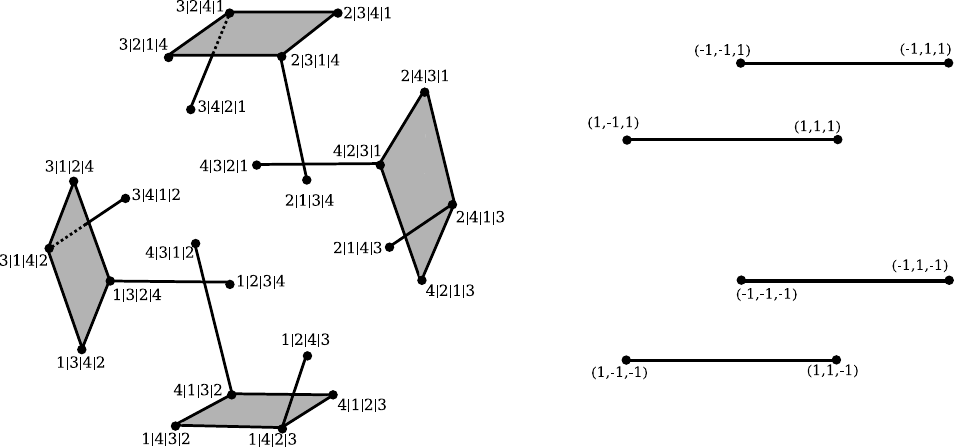}
    \caption{Permutohedral complex $\Perm(\K_3)$ and moment-angle 
    complex $\R_{\mathcal{L}_3}$.}
    \label{perm-1324}
\end{figure}

\smallskip

3. In general, the correspondence $\K \mapsto \mathcal{L}(\K)$ does not respect the combinatorial equivalence. For example, consider the simplicial complex $\K_3 = $
\begin{picture}(37, 20)
\multiput(10, 2)(15, 0){2}{\line(0, 1){15}}%
\multiput(10, 2)(0, 15){2}{\line(1, 0){15}}%
\put(10, 2){\circle*{3}} \put(25, 2){\circle*{3}} \put(10, 17){\circle*{3}} \put(25, 17){\circle*{3}}%
\put(2, 0){4} \put(30, 0){2} \put(2, 13){1} \put(30, 13){3}
\end{picture} that is obtained from $\K_2$ by changing the vertex order.
The permutohedral complex $\Perm(\K_3)$ is combinatorially equivalent to the complex $\Perm(\K_2)$, however, its image under the map $\rho$ is four edges of the cube.
They form the real moment-angle complex $\R_{\mathcal L_3}$ corresponding to the single vertex $\{2 \}$ on the set $[3]$. 
This complex is not combinatorially equivalent to $\mathcal{L}_2$.
The permutohedral complex $\Perm(\K_3)$ and the real moment-angle complex $\R_{\mathcal{L}_3}$ are shown in Figure~\ref{perm-1324}.
\end{example}

\end{document}